\documentclass[12pt,a4paper]{article}
\usepackage[top=3cm, bottom=3cm, left=3cm, right=3cm]{geometry}

\usepackage{amsmath,amsthm,amssymb,graphicx, multicol, array}
\usepackage{enumerate}
\usepackage{enumitem}
\usepackage{setspace}
\usepackage{xcolor}
\usepackage{currfile}
\usepackage{multicol}

\usepackage[bookmarks,colorlinks,breaklinks]{hyperref}  
\hypersetup{linkcolor=blue,citecolor=blue,filecolor=blue,urlcolor=blue}

\usepackage[hang,flushmargin]{footmisc} 
\usepackage{tabto}

\DeclareMathOperator{\li}{li}
\DeclareMathOperator{\ord}{ord}

\usepackage{tikz}
\usepackage{pgfplots}

\newtheorem{theorem}{Theorem}[section]
\newtheorem{lem}{Lemma}[section]

\newtheorem{exa}{Example}[section]

\newtheorem{dfn}{Definition}[section]
\newtheorem{exe}{Exercise}[section]

\newcommand{\F}{\mathbb{F}}
\newcommand{\tP}{\mathbb{P}}

\usepackage{fancyhdr}
\title{Simultaneous Primitive Root Values Of Polynomials Over Finite Fields}
\date{}
\author{N. A. Carella}

\begin{document}
\maketitle
\textbf{\textit{Abstract}:} Let $z\ne \pm1,w^2$ be a fixed integer, and let $f(t)\ne g(t)^2$ be a fixed polynomial over the integers. It is shown that the subset of primes $p\geq 2$ such that $z$ and $f(z)$ is a pair of simultaneous primitive roots modulo $p$ has nonzero density in the set of primes. The same analysis generalizes to \textit{admissible} $k$-tuple of polynomials $z$, $f_1(z)$, $f_2(z), \ldots$, $f_k(z)$, such that $f_i(z)\ne g_i(z)^2$, and $k\ll \log p$ is a small integer.

 \let\thefootnote\relax\footnote{ \today \date{} 
}

\tableofcontents

\section{Introduction}\label{S1210.000}
The investigation of the density of the subset of primes $p$ such that an \textit{admissible} integers $k$-tuple $z_1, z_2, \ldots, z_k$ are simultaneous primitive roots modulo $p$ has a well established conditional proof, see \cite{MK1983}. The investigation of the density of the subset of primes $p$ such that an \textit{admissible} polynomials $k$-tuple $z, f_1(z), f_2(z), \ldots, f_{k-1}(z)$ are simultaneous primitive roots modulo $p$, which is a topic of current interest, is simply a variation of the result proved by Mathews.    

\begin{dfn}   \label{conj1210.000}  {\normalfont A subset of polynomials $\{z, f_1(z), f_2(z), \ldots, f_{k-1}(z)\}$, such that $f_i(z)\ne g_i(z)^2$, is an \textit{admissible} $k$-tuple if there exists a prime $p\geq 2$ that supports one or more $k$-tuple of simultaneous primitive roots modulo $p$. Specifically, $$\ord_p z=\ord_p f_1(z)= \ord_p f_2(z)= \cdots =\ord_p f_{k-1}(z)=p-1.$$
}
\end{dfn} 

The simplest case consists of a fixed pair of simultaneous primitive root values of polynomials $z\ne \pm1,w^2$, and $f(z)\ne g(z)^2$. A related and conditional result for a fixed prime, and admissible quadratic polynomials is presented in \cite[Theorem 3]{BT2021}. In addition, some numerical data for small primes that support a pair of simultaneous primitive roots is compiled, opus citatum. Here, unconditional proofs are given for pair of simultaneous primitive root values of admissible polynomials of any degree. \\

Define the primes counting function 
\begin{equation} \label{eq1210.110}
\pi_{f}(x,z)=\# \left \{ p \leq x: \;\ord_p z=p-1\; \text{ and } \ord_p(f(z))=p-1\right \}.
\end{equation}

\begin{theorem}   \label{thm1210.200}  Let $z\ne 1, w^2$ be a fixed integer, and let $f(t)\ne g(t)^2$ be a fixed polynomial over the integers. Then, the number of primes for which $z$ and $f(z)$ is a pair of simultaneous primitive roots has the asymptotic formula
\begin{equation} \label{eq1213.210}
\pi_{f}(x,z)=\delta_f(z) \li(x)+O\left (\frac{x}{\log ^b x} \right ),
\end{equation}
where $x\geq 1$ is a large number, $\delta_f(z)>0$ is the density constant, and $b>1$ is constant.
\end{theorem} 

Define the primes counting function 
\begin{equation} \label{eq1210.310}
N(f,p)=\# \left \{ z<p: \;\ord_p z=p-1\; \text{ and } \ord_p(f(z))=p-1\right \}.
\end{equation}

\begin{theorem}   \label{thm1210.300}  Let $p\geq2$ be a fixed large prime, and let $f(t)\ne g(t)^2$ be a fixed polynomial over the integers. Then, the number of integers $z<p$ for which $z$ and $f(z)$ is a pair of simultaneous primitive roots modulo $p$ has the asymptotic formula
\begin{equation} \label{eq1213.310}
N_{f}(p)= c(f,p)\left (\frac{\varphi(p-1)}{p-1}  \right )^2p +O\left(p^{1-2\varepsilon}\right),
\end{equation}
where $c(f,p)>0$ is the density constant, and $\varepsilon>0$ is a small number.
\end{theorem} 

The preliminary notation, definitions, background results are discussed in Section \ref{S3242} to Section \ref{S8799}. The proof of Theorem \ref{thm1210.200} appears in Section \ref{S1213.000}, and the proof of Theorem \ref{thm1210.300} appears in Section \ref{S8240.000}.

\section{Representation of the Characteristic Function} \label{S3242}
The result in Lemma \ref{lem3242.300} provides a divisor-free of the representation of the characteristic function of primitive elements in finite fields $\F_p$. 

\begin{lem} \label{lem3242.300}
Let \(p\geq 2\) be a prime, and let \(\tau\) be a primitive root mod \(p\). If \(u\in\mathbb{F}_p\) is a nonzero element, then
\begin{equation}\label{eq3242.300}
\Psi (u)=\sum _{\gcd (n,p-1)=1} \frac{1}{p}\sum _{0\leq k\leq p-1} e^{i2 \pi \left (\tau ^n-u\right)k/p}
=\left \{
\begin{array}{ll}
1 & \text{ if } \ord_p (u)=p-1,  \\
0 & \text{ if } \ord_p (u)\neq p-1. \\
\end{array} \right .\nonumber
\end{equation}
\end{lem}

\begin{proof} A complete proof appears in  Lemma 2.1 in \cite{CN2021}.
\end{proof}

\section{Evaluation Of The Main Term $M(x)$} \label{S1215.000}
The precise evaluation of the main term $M(x)$ occurring in the proof of Theorem \ref{thm1210.200} is recorded here. The symbol $\li(x)$ denotes the logarithm integral.

\begin{lem}  \label{lem1215.100}{\normalfont (\cite[Lemma 4]{VR1973})} Let \(x\geq 1\) be a large number, and let $k\geq 1$ be a fixed constant. Then,
\begin{equation} \label{eq1215.110}
\sum_{p\leq x} \left (\frac{\varphi(p-1)}{p-1}  \right )^k=a_k\li(x)+O\left(\frac{x}{\log ^bx}\right) \nonumber,
\end{equation} 
where $p$ varies over the primes, and $a_k>0$, and $b>1$ are constants.
\end{lem}

\begin{lem}  \label{lem1215.200} If \(x\geq 1\) is a large number, then,
\begin{equation} \label{eq1215.210}
\sum_{p\leq x} \left (\frac{1}{p}\sum_{\gcd(k,p-1)=1} 1\right )\left (\frac{1}{p}\sum_{\gcd(n,p-1)=1} 1\right )=a_2\li(x)+O\left(\frac{x}{\log ^bx}\right) \nonumber,
\end{equation} 
where $a_2>0$, and $b>1$ are constants.
\end{lem}

\begin{proof} Evaluate the inner sums and apply Lemma \ref{lem1215.100} to obtain
\begin{eqnarray} \label{eq1215.320}  
M(x)&=&\sum_{p\leq x} \left (\frac{1}{p}\sum_{\gcd(k,p-1)=1} 1\right )\left (\frac{1}{p}\sum_{\gcd(n,p-1)=1} 1\right )\\
&=&\sum_{p\leq x} \left (\frac{\varphi(p-1)}{p-1}  \right )^2\left (1-\frac{1}{p}  \right )^2\nonumber\\
&=&a_2\li(x)+O\left(\frac{x}{\log ^bx}\right) \nonumber,
\end{eqnarray}  
as claimed. 
\end{proof}

The constant for $k=1$, which is the average density of primitive roots modulo $p$, is given by $a_1=\prod_{p \geq 2 } \left(1-1/p(p-1)\right)= 0.37395581361920228805\ldots,$ 
see \cite{SP1969}. The numerical value for $a_2$, which is for the average density of a pair of simultaneous primitive roots, is not available in the literature.

\section{Evaluation Of The Main Term $M(f,p)$} \label{S8215.000}
The precise evaluation of the main term $M(f,p)$ occurring in the proof of Theorem \ref{thm1210.300} is recorded here. 

\begin{lem}  \label{lem8215.200} If \(x\geq 1\) is a large number, then,
\begin{equation} \label{eq8215.210}
\sum_{z \leq p} \left (\frac{1}{p}\sum_{\gcd(k,p-1)=1} 1\right )\left (\frac{1}{p}\sum_{\gcd(n,p-1)=1} 1\right )=\left (\frac{\varphi(p-1)}{p-1}  \right )^2p+O\left(1\right) \nonumber.
\end{equation} 
\end{lem}

\begin{proof} Rearrange the triple sum as
\begin{eqnarray} \label{eq8215.320}  
M(f,p)&=&\sum_{z \leq p} \left (\frac{1}{p}\sum_{\gcd(k,p-1)=1} 1\right )\left (\frac{1}{p}\sum_{\gcd(n,p-1)=1} 1\right )\\
&=&\sum_{z \leq p} \left (\frac{\varphi(p-1)}{p-1}  \right )^2\left (1-\frac{1}{p}  \right )^2\nonumber\\
&=&\left (\frac{\varphi(p-1)}{p-1}  \right )^2p+O\left(1\right) \nonumber,
\end{eqnarray}  
as claimed. 
\end{proof}

\section{Estimate For The Error Terms $E_i(x)$} \label{S2799}
The upper bounds for the error terms $E_0(x)$, $E_1(x)$, and $E_2(x)$ in the proof of Theorem \ref{thm1210.200} are recorded here.

\begin{lem} \label{lem2799.200} If $x$ is a large number, then
\begin{equation}\label{eq2799.200}
E_0(x)=\sum_{x\leq p\leq 2x} \frac{1}{p}\sum_{\substack{0\leq  a< p\\\gcd(n,p-1)=1}}e^{\frac{i2\pi a(\tau^{n}-u)}{p}} \cdot \frac{1}{p} \sum_{\substack{0<\leq b< p\\gcd(m,p-1)=1}}e^{\frac{i2\pi b(\tau^{m}-v)}{p}}
=0\nonumber.
\end{equation}
\end{lem}
\begin{proof} By hypothesis, $\ord_pu \ne p-1$, so the first finite sum
\begin{equation}\label{eq2799.220}
\sum_{\substack{0\leq  a< p\\\gcd(n,p-1)=1}}e^{\frac{i2\pi a(\tau^{n}-u)}{p}} 
=0
\end{equation}
is a geometric series, which vanishes.
\end{proof}

\begin{lem} \label{lem2799.300} If $x$ is a large number, then
\begin{equation}\label{eq2799.300}
E_1(x)=\sum_{x\leq p\leq 2x} \frac{1}{p}\sum_{\substack{0< b< p\\gcd(m,p-1)=1}}e^{\frac{i2\pi b(\tau^{m}-v)}{p}} \cdot \frac{1}{p} \sum_{\substack{0\leq b< p\\gcd(m,p-1)=1}}e^{\frac{i2\pi b(\tau^{em}-v)}{p}}
=0\nonumber.
\end{equation}
\end{lem}
\begin{proof} By hypothesis, $\ord_pv \ne p-1$, so the second finite sum
\begin{equation}\label{eq2799.302}
\sum_{\substack{0\leq  b< p\\\gcd(m,p-1)=1}}e^{\frac{i2\pi b(\tau^{em}-v)}{p}} 
=0
\end{equation}
is a geometric series, which vanishes.
\end{proof}

\begin{lem} \label{lem2799.400} If $x$ is a large number, then
\begin{equation}\label{eq2799.400}
E_2(x)=\sum_{x\leq p\leq 2x} \frac{1}{p}\sum_{\substack{0< a< p\\gcd(n,p-1)=1}}e^{\frac{i2\pi a(\tau^{n}-u)}{p}} \cdot \frac{1}{p} \sum_{\substack{0< b< p\\\gcd(m,p-1)=1}}e^{\frac{i2\pi b(\tau^{em}-v)}{p}}
\ll x^{1-2\varepsilon}\nonumber.
\end{equation}
\end{lem}
\begin{proof} To compute an upper bound, define the exponential sum
\begin{equation}\label{eq2799.402}
T(u,p)=\sum_{\substack{0< a< p\\gcd(n,p-1)=1}}e^{\frac{i2\pi a(\tau^{n}-u)}{p}}.
\end{equation}
Now, apply Lemma \ref{lem3700.200} to each factor $T(u,p)$ and $T(v,p)$ to obtain the followings.
\begin{eqnarray}\label{eq2799.404}
E_2(x)&=&\sum_{x\leq p\leq 2x} \left (\frac{1}{p}\cdot T(u,p) \right )\cdot \left (\frac{1}{p}\cdot T(v,p)\right )\\
&\ll&\sum_{x\leq p\leq 2x} \left (\frac{1}{p}\cdot p^{1-\varepsilon} \right )\cdot \left (\frac{1}{p}\cdot p^{1-\varepsilon}\right ) \nonumber.
\end{eqnarray}
Take an upper bound, and apply the prime number theorem:  
\begin{eqnarray}\label{eq2799.406}
E_2(x)
&\ll& \frac{1}{x^{2\varepsilon}}\sum_{x\leq p\leq 2x} 1\\
&\ll& x^{1-2\varepsilon}\nonumber,
\end{eqnarray}
as $x\to \infty$.
\end{proof}

\begin{lem} \label{lem3700.200} Given a small number $\varepsilon>0$. Let $p$ be a large prime number, and let $\tau$ be a primitive root modulo $p$. Then,
\begin{equation}\label{eq3700.200}
\sum_{\substack{0< a< p\\\gcd(n,p-1)=1}}e^{\frac{i2\pi a(\tau^{n}-u)}{p}}
\ll p^{1-\varepsilon}.
\end{equation}
\end{lem}
\begin{proof} A complete detailed proof is given in \cite[Lemma 8.1]{CN2021}. 
\end{proof}

\section{Estimate For The Error Terms $E_i(f,p)$} \label{S8799}
The upper bounds for the error terms $E_0(f,p)$, $E_1(f,p)$, and $E_2(f,p)$ in the proof of Theorem \ref{thm1210.300} are recorded here.

\begin{lem} \label{lem8799.200} If $x$ is a large number, then
\begin{equation}\label{eq8799.200}
E_0(f,p)=\sum_{z\leq p} \frac{1}{p}\sum_{\substack{0\leq  a< p\\\gcd(n,p-1)=1}}e^{\frac{i2\pi a(\tau^{n}-u)}{p}} \cdot \frac{1}{p} \sum_{\substack{0<\leq b< p\\gcd(m,p-1)=1}}e^{\frac{i2\pi b(\tau^{m}-v)}{p}}
=0\nonumber.
\end{equation}
\end{lem}
\begin{proof} By hypothesis, $\ord_pu \ne p-1$, so the first finite sum
\begin{equation}\label{eq8799.220}
\sum_{\substack{0\leq  a< p\\\gcd(n,p-1)=1}}e^{\frac{i2\pi a(\tau^{n}-u)}{p}} 
=0
\end{equation}
is a geometric series, which vanishes.
\end{proof}
\begin{lem} \label{lem8799.300} If $x$ is a large number, then
\begin{equation}\label{eq8799.300}
E_1(f,p)=\sum_{z\leq p} \frac{1}{p}\sum_{\substack{0< b< p\\gcd(m,p-1)=1}}e^{\frac{i2\pi b(\tau^{m}-v)}{p}} \cdot \frac{1}{p} \sum_{\substack{0\leq b< p\\gcd(m,p-1)=1}}e^{\frac{i2\pi b(\tau^{em}-v)}{p}}
=0\nonumber.
\end{equation}
\end{lem}
\begin{proof} By hypothesis, $\ord_pv \ne p-1$, so the second finite sum
\begin{equation}\label{eq8799.302}
\sum_{\substack{0\leq  b< p\\\gcd(m,p-1)=1}}e^{\frac{i2\pi b(\tau^{em}-v)}{p}} 
=0
\end{equation}
is a geometric series, which vanishes.
\end{proof}

\begin{lem} \label{lem8799.400} If $x$ is a large number, then
\begin{equation}\label{eq8799.400}
E_2(f,p)=\sum_{z\leq p} \frac{1}{p}\sum_{\substack{0< a< p\\gcd(n,p-1)=1}}e^{\frac{i2\pi a(\tau^{n}-u)}{p}} \cdot \frac{1}{p} \sum_{\substack{0< b< p\\\gcd(m,p-1)=1}}e^{\frac{i2\pi b(\tau^{em}-v)}{p}}
\ll p^{1-2\varepsilon}\nonumber.
\end{equation}
\end{lem}
\begin{proof} To compute an upper bound, define the exponential sum
\begin{equation}\label{eq8799.402}
T(u,p)=\sum_{\substack{0< a< p\\gcd(n,p-1)=1}}e^{\frac{i2\pi a(\tau^{n}-u)}{p}}.
\end{equation}
Now, apply Lemma \ref{lem3700.200} to each factor $T(u,p)$ and $T(v,p)$ to obtain the followings.
\begin{eqnarray}\label{eq8799.404}
E_2(f,p)&=&\sum_{z\leq p} \left (\frac{1}{p}\cdot T(u,p) \right )\cdot \left (\frac{1}{p}\cdot T(v,p)\right )\\
&\ll&\sum_{z\leq p} \left (\frac{1}{p}\cdot p^{1-\varepsilon} \right )\cdot \left (\frac{1}{p}\cdot p^{1-\varepsilon}\right ) \nonumber.
\end{eqnarray}
Take an upper bound, and apply the prime number theorem:  
\begin{eqnarray}\label{eq8799.406}
E_2(f,p)
&\ll& \frac{1}{p^{2\varepsilon}}\sum_{z\leq p} 1\\
&\ll& p^{1-2\varepsilon}\nonumber,
\end{eqnarray}
as $p\to \infty$.
\end{proof}

\section{Proof Of Theorem 1.1} \label{S1213.000}
Given a fixed integer $z\ne \pm 1, w^2$, and a fixed polynomial $f(t)\ne g(t)^2$, define the primes counting functions 
\begin{equation} \label{eq1213.100}
\pi(x)=\# \{ p \leq x: p \text{ is prime} \},
\end{equation}
and
\begin{equation} \label{eq1213.110}
\pi_{f}(x,z)=\# \left \{ p \leq x: \;\ord_p z=p-1\; \text{ and } \ord_p(f(z))=p-1\right \}.
\end{equation}

The density of the subset of primes 
\begin{equation} \label{eq1213.115}
\mathcal{D}=\{ p \in \tP:\ord_p z=p-1\; \text{ and } \ord_p(f(z))=p-1 \},
\end{equation}
is defined by the limit
\begin{equation} \label{eq1213.120}
\delta_f(z)=\lim_{x \to \infty} \frac{\pi_{f}(x,z)}{\pi(x)}=c_f(z)a_2,
\end{equation}
where $c_f(z)\geq 0$ is a correction factor(rational number), and $a_2>0$ is defined in Lemma \ref{lem1215.100}. The proof below shows that for any given fixed pair $z\ne\pm 1, w^2$ and $f(t)\ne g(t)^2$, the subset $\mathcal{D}$ of primes has nonzero density in the set of primes $\tP=\{2,3,5,7, \ldots \}$.

\begin{proof} {\bf (Theorem \ref{thm1210.200})}. Let \(x\geq x_0\) be a large number, and let $p\geq x$ be a prime. The characteristic function of a pair of simultaneous primitive roots $z$ and $f(z)$ in $\F_p$ has the exact formula
\begin{equation} \label{eq1213.200}
 \Psi (z)\Psi (f(z)),
\end{equation}
where $\Psi (z) $ is characteristic function of primitive root in $\F_p$. Now suppose that there are no primes $p\geq x$ that support any pair of simultaneous primitive roots $z$ and $f(z)$. Summing of the exact formula \eqref{eq1213.200} over the short interval $[x,2x]$ returns the nonexistence equation
\begin{equation} \label{eq1213.210}
	0=\sum_{x \leq p\leq 2x} \Psi (z)\Psi (f(z)).
\end{equation}
Replacing the characteristic function, Lemma \ref{lem3242.300}, and expanding the nonexistence equation \eqref{eq1213.210} yield
\begin{eqnarray} \label{eq1213.220}
0&=&\sum_{x \leq p\leq 2x} \Psi (z)\Psi (f(z)) \\ 
	&=&\sum_{x \leq p\leq 2x} \left (\frac{1}{p}\sum_{\gcd(k,p-1)=1,} \sum_{ 0\leq a\leq p-1} e^{i2 \pi \left (\tau ^k-z\right)a/p} \right ) \nonumber\\
&&\hskip 1 in \times \left (\frac{1}{p}\sum_{\gcd(n,p-1)=1,} \sum_{ 0\leq b\leq p-1} e^{i2 \pi \left (\tau ^n-f(z)\right)b/p} \right )\nonumber\\
&=&M(x) + E_0(x)+ E_1(x)+ E_2(x)\nonumber.
\end{eqnarray} 

\begin{enumerate}
\item The main term $M(x)$ is determined by a finite sum over the trivial additive characters pair $\psi_a(t) =e^{i 2\pi  at/p}=1 $ at $a=0$, and $\psi_b(t) =e^{i 2\pi  bt/p}=1 $ at $b=0$. This is computed in Lemma \ref{lem1215.200}.\\

\item The error term $E_0(x)$ is determined by a finite sum over the nontrivial additive characters pair $\psi_a(t) =e^{i 2\pi  at/p}=1 $ at $a=0$, and $\psi_b(t) =e^{i 2\pi  bt/p}\ne1 $ at $b \ne 0$. This is estimated in Lemma \ref{lem2799.200}.\\

\item The error term $E_1(x)$ is determined by a finite sum over the nontrivial additive characters pair $\psi_a(t) =e^{i 2\pi  at/p}\ne 1 $ at $a\ne 0$, and $\psi_b(t) =e^{i 2\pi  bt/p}=1 $ at $b = 0$. This is computed in Lemma \ref{lem2799.300}.\\

\item The error term $E_2(x)$ is determined by a finite sum over the nontrivial additive characters pair $\psi_a(t) =e^{i 2\pi  at/p}\ne 1$ at $a\ne 0$, and $\psi_b(t) =e^{i 2\pi bt/p}\ne 1 $ at $b \ne 0$. This is computed in Lemma \ref{lem2799.400}.\\
\end{enumerate}

Applying Lemma \ref{lem1215.200} to the main term, and Lemmas \ref{lem2799.200} to \ref{lem2799.400} to the error terms yield
	\begin{eqnarray} \label{eq3292.960}
\sum_{x \leq p\leq 2x} \Psi (z)\Psi (f(z)) 
&=&M(x) + E_0(x)+ E_1(x)+ E_2(x)\\
&=&\delta_f(z)\left (\li(2x)-\li(x) \right )+O\left(\frac{x}{\log^b x}\right)+O\left( x^{1-\varepsilon} \right) \nonumber\\
&=&\delta_f(z)\left (\li(2x)-\li(x)\right )+O\left( \frac{x }{\log^b x} \right)  \nonumber ,
	\end{eqnarray} 
	where $\delta_f(z)\geq 0$ is defined in \eqref{eq1213.115}. \\ 
	
Since the density $\delta_f(z)> 0$ for a pair $z \ne \pm 1, w^2$, and $f(z)\ne g(z)^2$, see \cite[Theorem 13.1]{MK1983}, and the difference of logarithm integrals
\begin{equation}
 \li(2x)-\li(x)\gg \frac{x}{\log x}>0   
\end{equation}
for large $x\gg1$, the expression
\begin{eqnarray} \label{eq3292.970}
\sum_{x \leq p\leq 2x} \Psi (z)\Psi (f(z))
&=&\delta_f(z)\left (\li(2x)-\li(x)\right )+O\left( \frac{x }{\log^b x} \right)  \\
	&\gg& \frac{x}{\log x}\nonumber\\
    &>&0,\nonumber
	\end{eqnarray} 
is false for all sufficiently large numbers $x\gg1$, and contradicts the hypothesis \eqref{eq1213.210}. Therefore, the short interval $[x,2x]$ contains primes $p\geq x$ such that $z \ne \pm 1, w^2$, and $f(z)\ne g(z)^2$ are simultaneous primitive roots modulo $p$. 
\end{proof}

\section{Proof Of Theorem 1.2} \label{S8240.000}
As the main requirement for a primitive root value over the integers, $z\ne\pm 1, w^2$, the main requirement for a primitive root valued polynomial is $f(t)\ne g(t)^2$, see Example \ref{exa4545.300} for numerical evidence.\\

For a fixed prime $p\geq 2$, and a fixed polynomial $f(t)\ne g(t)^2$, the counting function for the pair of simultaneous primitive roots $z>1$ and $f(z)>1$ is defined by 
\begin{equation} \label{eq8240.110}
\mathcal{S}_p=\{ z<p :\ord_p z=p-1\; \text{ and } \ord_p(f(z))=p-1 \},
\end{equation}
and the average density over the set of primes $\tP=\{2,3,5,7, \ldots \}$ is defined by the limit
\begin{equation} \label{eq8240.120}
c(f,p)=\lim_{p \to \infty} \frac{\mathcal{S}_p}{p}.
\end{equation}

\begin{proof} {\bf (Theorem \ref{thm1210.300})}. Let $p\geq 2$ be a prime, and let $z<p$. The characteristic function of a pair of simultaneous primitive roots $z$ and $f(z)$ in $\F_p$ has the exact formula
\begin{equation} \label{eq8240.200}
 \Psi (z)\Psi (f(z)),
\end{equation}
where $\Psi (z) $ is characteristic function of primitive root in $\F_p$. Now suppose that a large fixed prime $p\geq 2$ does not support any pair of simultaneous primitive roots $z$ and $f(z)$ modulo $p$. Summing of the exact formula \eqref{eq8240.200} over the elements $z<p$ returns the nonexistence equation
\begin{equation} \label{eq8240.210}
	0=\sum_{z< p} \Psi (z)\Psi (f(z)).
\end{equation}
Replacing the characteristic function, Lemma \ref{lem3242.300}, and expanding the nonexistence equation \eqref{eq8240.210} yield
\begin{eqnarray} \label{eq8240.220}
0&=&\sum_{z< p}\Psi (z)\Psi (f(z)) \\ 
	&=&\sum_{z< p} \left (\frac{1}{p}\sum_{\gcd(k,p-1)=1,} \sum_{ 0\leq a\leq p-1} e^{i2 \pi \left (\tau ^k-z\right)a/p} \right ) \nonumber\\
&&\hskip 1 in \times \left (\frac{1}{p}\sum_{\gcd(n,p-1)=1,} \sum_{ 0\leq b\leq p-1} e^{i2 \pi \left (\tau ^n-f(z)\right)b/p} \right )\nonumber\\
&=&M(f,p) + E_0(f,p)+ E_1(f,p)+ E_2(f,p)\nonumber.
\end{eqnarray} 

\begin{enumerate}
\item The main term $M(f,p)$ is determined by a finite sum over the trivial additive characters pair $\psi_a(t) =e^{i 2\pi  at/p}=1 $ at $a=0$, and $\psi_b(t) =e^{i 2\pi  bt/p}=1 $ at $b=0$. This is computed in Lemma \ref{lem8215.200}.\\

\item The error term $E_0(f,p)$ is determined by a finite sum over the nontrivial additive characters pair $\psi_a(t) =e^{i 2\pi  at/p}=1 $ at $a=0$, and $\psi_b(t) =e^{i 2\pi  bt/p}\ne1 $ at $b \ne 0$. This is estimated in Lemma \ref{lem8799.200}.\\

\item The error term $E_1(f,p)$ is determined by a finite sum over the nontrivial additive characters pair $\psi_a(t) =e^{i 2\pi  at/p}\ne 1 $ at $a\ne 0$, and $\psi_b(t) =e^{i 2\pi  bt/p}=1 $ at $b = 0$. This is computed in Lemma \ref{lem8799.300}.\\

\item The error term $E_2(f,p)$ is determined by a finite sum over the nontrivial additive characters pair $\psi_a(t) =e^{i 2\pi  at/p}\ne 1$ at $a\ne 0$, and $\psi_b(t) =e^{i 2\pi bt/p}\ne 1 $ at $b \ne 0$. This is computed in Lemma \ref{lem8799.400}.\\
\end{enumerate}

Applying Lemma \ref{lem8215.200} to the main term, and Lemmas \ref{lem8799.200} to \ref{lem8799.400} to the error terms yield
\begin{eqnarray} \label{eq8240.960}
\sum_{z< p} \Psi (z)\Psi (f(z)) 
&=&M(f,p) + E_0(f,p)+ E_1(f,p)+ E_2(f,p)\\
&=&c(f,p)\left (\frac{\varphi(p-1)}{p-1}  \right )^2p +O\left(p^{1-2\varepsilon}\right) \nonumber ,
\end{eqnarray} 
where $c(f,z)\geq 0$ is defined in \eqref{eq8240.120}. \\ 
	
Since there are pairs $z \ne \pm 1, w^2$, and $f(z)\ne g(z)^2$ of simultaneous primitive roots, see \cite[Theorem 13.1]{MK1983}, the density $c(f,z)> 0$. Moreover, the term 
\begin{equation}
 \left (\frac{\varphi(p-1)}{p-1}  \right )^2\gg \frac{1}{(\log \log p)^2}>0   
\end{equation}
for large $p\gg1$. Consequently, the expression
\begin{eqnarray} \label{eq8240.970}
\sum_{z< p} \Psi (z)\Psi (f(z))
&=&c(f,p)\left (\frac{\varphi(p-1)}{p-1}  \right )^2p +O\left(p^{1-2\varepsilon}\right), \\
	&\gg& \frac{p}{(\log \log p)^2}\nonumber\\
    &>&0,\nonumber
	\end{eqnarray} 
is false for all sufficiently large primes $p\gg1$, and contradicts the hypothesis \eqref{eq8240.210}. Therefore, the finite field $\F_p$ contains a small subset $\mathcal{S}_p$ of elements $z\in \F_p$ such that $z>1$, and $f(z)>1$  are simultaneous primitive roots for all sufficiently large primes $p\geq2$. 
\end{proof}

\section{Some Numerical Examples}\label{exa4545}
A few small cases were computed to determine the accuracy of Theorem \ref{thm1210.300}, and to establish some notation. This simple experiment shows that irreducible polynomials seem to generate simpler spectra (distribution patterns) than the reducible polynomials, compare Figure \ref{f4545.200} and Figure \ref{f4545.400}.

\newpage
\begin{exa}\label{exa4545.100}{\normalfont The preliminary statistics for the finite field $\F_{97}$ and the polynomial $f(t)=t^2+1$ are listed here.

\begin{enumerate}
\item $\displaystyle \varphi(p-1)=32$, the total number of primitive roots, listed in Table \ref{t4545.100}. \\
\item $\displaystyle \left (\frac{\varphi(p-1)}{p-1}\right )^2p= \frac{24832}{2401}\approx10.34$, 
an estimate of the asymptotic (main term) for total number of pairs of simultaneous primitive roots $z, f(z)$, the actual number is $4$, see Table \ref{t4545.100}. \\
\end{enumerate}
Table \ref{t4545.100} lists all the primitive roots $z$ modulo $p=97$ and the corresponding simultaneous pairs, if a pair $z$ and $f(z)$ exists, and Figure \ref{f4545.100} displays the spectrum of simultaneous pairs of primitive roots.

}
\end{exa}

\begin{table}[h!]
\centering
\caption{Data for the prime finite field $\F_{97}$} \label{t4545.100}
\begin{tabular}{c|c|c| c|c|c|c| c}
$z$&$f(z)=z^2+1$&$z$&$f(z)=z^2+1$&$z$&$f(z)=z^2+1$&$z$&$f(z)=z^2+1$\\
\hline
5&  26&     29   &0  &60&  0&   90   &0\\
7&    0&    37   & 0 &68  & 0&  92     &26\\
10&   0&    38   &87 &71  &0&& \\
13&   0&    39   &0  &74   &0&& \\
14&   0&    40   &0  &76&  0&     &\\
15&  0&     41   &0  &80& 0 &      &\\
17&  0&     56   & 0 &82& 0 &      &\\
21&  0&     57   &0  &83& 0 &      &\\
23&  0&     58   &0  &84& 0 &      &\\
26&  0&     59   &87  &87& 0 &     &\\

\end{tabular}
\end{table}

\begin{figure}[h]
\caption{The spectrum of the pair $(z,f(z))$ of simultaneous primitive roots in $\F_{97}$} \label{f4545.100}\centering%
  \begin{tikzpicture}
	\begin{axis}[
		xlabel=$z$,
		ylabel=$f(z)$,
width=0.95\textwidth,
       height=0.5\textwidth		]
	\addplot[only marks] coordinates {
		(5,26)
		(7,0)
		(10,0)
		(14,0)
		(15,0)
		(17,0)
		(21,0)
		(23,0)
		(26,0)
		(29,0)
		(37,0)
		(38,87)
		(39,0)(40,0)(41,0)(56,0)(58,0)
		(59,87)
		(60,0)(68,0)(71,0)(74,0)(76,0)(80,0)(82,0)(83,0)(84,0)(87,0)(90,0)
		(92,26)
		};
	\end{axis}
\end{tikzpicture}	
\end{figure}
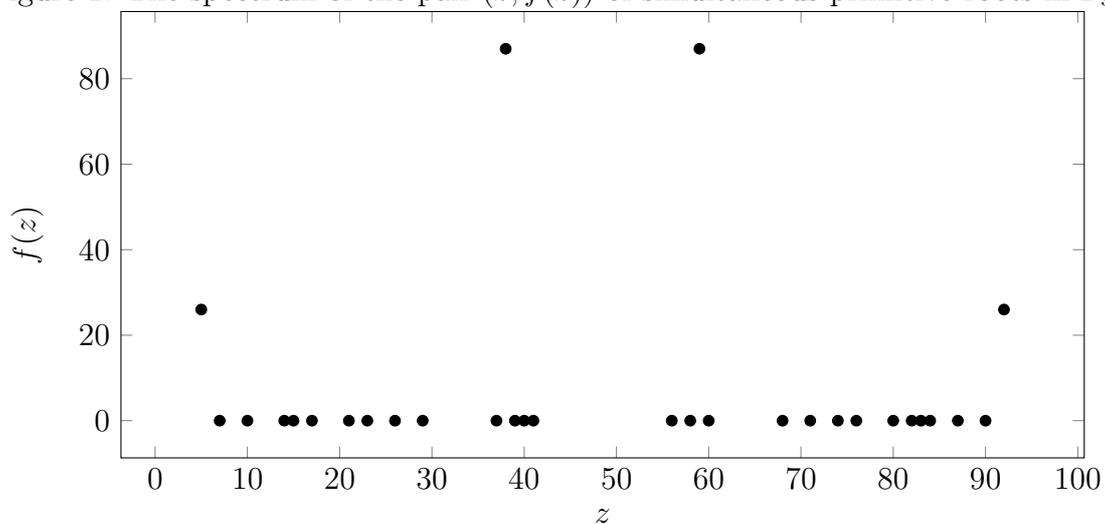

\begin{exa}\label{exa4545.200}{\normalfont The preliminary statistics for the finite field $\F_{101}$ and the polynomial $f(t)=t^2+1$ are listed here.

\begin{enumerate}
\item $\displaystyle \varphi(p-1)=40$, the total number of primitive roots, listed in Table \ref{t4545.200}. \\
\item $\displaystyle \left (\frac{\varphi(p-1)}{p-1}\right )^2p= \frac{404}{25}\approx16.16$, 
an estimate of the asymptotic (main term) for the total number of pairs of simultaneous primitive roots $z, f(z)$, the actual number is $12$, see Table \ref{t4545.200}. \\
\end{enumerate}
Table \ref{t4545.200} lists all the primitive roots $z$ modulo $p=101$ and the corresponding simultaneous pairs, if a pair $z$ and $f(z)$ exists, and Figure \ref{f4545.200} displays the spectrum of simultaneous pairs of primitive roots.

}
\end{exa}

\begin{table}[h!]
\centering
\caption{Data for the prime finite field $\F_{101}$} \label{t4545.200}
\begin{tabular}{c|c|c| c|c|c|c| c}
$z$&$f(z)=z^2+1$&$z$&$f(z)=z^2+1$&$z$&$f(z)=z^2+1$&$z$&$f(z)=z^2+1$\\
\hline
2&  0&     28   &0  &51&  0&   74   &0\\
3&    0&   29   &34 &53  & 83&  75     &0\\
7&  50&    34   &46 &55  &0&    83&0 \\
8&   0&    35   &0  &59  &48&   86&0 \\
11&   0&   38   &0  &61&  86 &  89&0\\
12&  0&    40   &86  &63& 0 &   90   &0\\
15&  0&    42   & 48 &66& 0 &   93   &0\\
18&  0&    46   &0  &67& 46 &   94   &50\\
26&  0&    48   &83  &72& 34 &   98   &0\\
27&  0&    50   &0  &73& 0 &   99  &0\\

\end{tabular}
\end{table}

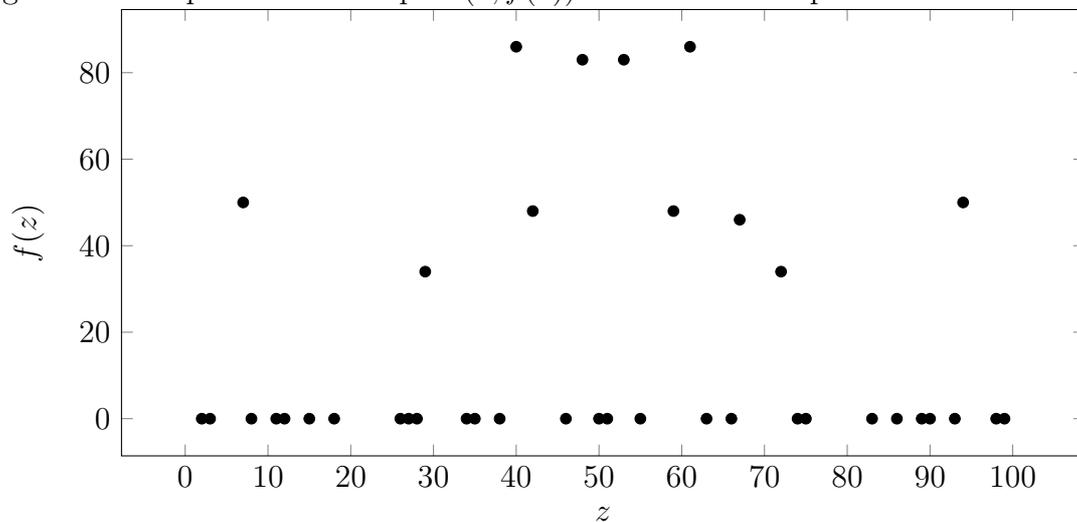
\begin{figure}[h]
\caption{The spectrum of the pair $(z,f(z))$ of simultaneous primitive roots in $\F_{101}$} \label{f4545.200}\centering%
  \begin{tikzpicture}
	\begin{axis}[
		xlabel=$z$,
		ylabel=$f(z)$,
width=0.95\textwidth,
       height=0.5\textwidth		]
	\addplot[only marks] coordinates {
		(2,0)(3,0)(7,50)(8,0)(11,0)(12,0)(15,0)(18,0)
		(26,0)(27,0)(28,0)(29,34)
		(34,0)(35,0)(38,0)(40,86)(42,48)(46,0)(48,83)
		(50,0)(51,0)(53,83)(55,0)(59,48)(61,86)
		(63,0)(66,0)(67,46)(72,34)(74,0)(75,0)(83,0)
		(86,0)
		(89,0)
		(90,0)(93,0)
		(94,50)(98,0)(99,0)
		};
	\end{axis}
\end{tikzpicture}	
\end{figure}

\begin{exa}\label{exa4545.300}{\normalfont The preliminary statistics for the finite field $\F_{127}$ and the polynomial $f(t)=(t+2)(t+1)^2$ are listed here.

\begin{enumerate}
\item $\displaystyle \varphi(p-1)=36$, the total number of primitive roots, listed in Table \ref{t4545.300}. \\
\item $\displaystyle \left (\frac{\varphi(p-1)}{p-1}\right )^2p= \frac{508}{49}\approx10.37$, 
an estimate of the asymptotic (main term) for the total number of pairs of simultaneous primitive roots $z, f(z)$, the actual number is $9$, see Table \ref{t4545.300}. \\
\end{enumerate}
Table \ref{t4545.300} lists all the primitive roots $z$ modulo $p=127$ and the corresponding simultaneous pairs, if a pair $z$ and $f(z)=(z+2)(z+1)^2$ exists, and Figure \ref{f4545.300} displays the spectrum of simultaneous pairs of primitive roots.

}
\end{exa}

\begin{table}[h!]
\centering
\caption{Data for the prime finite field $\F_{127}$} \label{t4545.300}
\begin{tabular}{c|c|c| c|c|c|c| c}
$z$&$f(z)=(z+2)(z+1)^2$&$z$&$f(z)$&$z$&$f(z)$&$z$&$f(z)$\\
\hline
3&  0&     46   &114    &83&  0&   109   &0\\
6&    0&   48   &0    &85  & 0&  110     &97\\
7&  0&    53   &106     &86  &0&    112&0 \\
12&   0&    55   &0   &91  &6&   114&67 \\
14&   0&   56   &101   &92&  0 &  116&116\\
23&  0&    57   &0   &93& 0 &   118   &0\\
29&  0&    58   &0    &96& 0 &     &\\
39&  0&    65   &6  &97& 0 &     &\\
43&  0&    67   &0  &101& 0 &      &\\
45&  0&    78   &43  &106& 0 &     &\\

\end{tabular}
\end{table}

\begin{figure}[h]
\caption{The spectrum of the pair $(z,f(z))$ of simultaneous primitive roots in $\F_{127}$} \label{f4545.300}\centering%
  \begin{tikzpicture}
	\begin{axis}[
		xlabel=$z$,
		ylabel=$f(z)$,
width=0.95\textwidth,
       height=0.5\textwidth		]
	\addplot[only marks] coordinates {
		(3,0)(6,0)(7,0)(12,0)(14,0)(23,0)(29,0)(39,0)
		(43,0)(45,0)(46,114)(48,0)
		(53,106)(55,0)(56,101)(57,0)(58,0)(65,6)(67,0)(78,43)
		(83,0)(85,0)(86,0)(91,6)(92,0)(93,0)
		(96,0)(97,0)(101,0)(106,0)(109,0)(110,97)(112,0)
		(114,67)(116,116)(118,0)
		};
	\end{axis}
\end{tikzpicture}	
\end{figure}
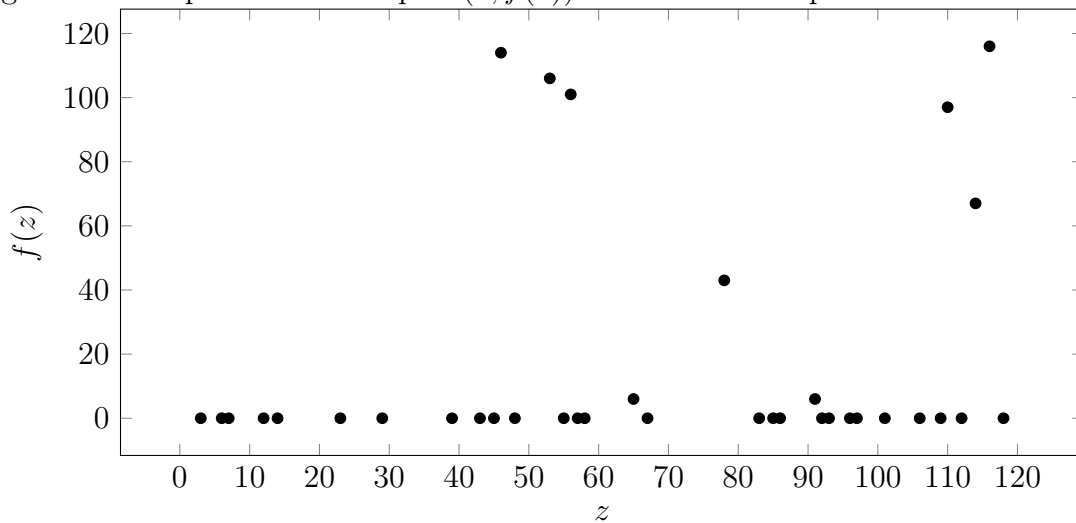

\newpage
\begin{exa}\label{exa4545.400}{\normalfont The preliminary statistics for the finite field $\F_{89}$ and the polynomial $f(t)=(t+2)(t+1)^2$ are listed here.

\begin{enumerate}
\item $\displaystyle \varphi(p-1)=40$, the total number of primitive roots, listed in Table \ref{t4545.400}. \\
\item $\displaystyle \left (\frac{\varphi(p-1)}{p-1}\right )^2p= \frac{2225}{121}\approx18.39$, 
an estimate of the asymptotic (main term) for the total number of pairs of simultaneous primitive roots $z, f(z)$, the actual number is $18$, see Table \ref{t4545.400}. \\
\end{enumerate}
Table \ref{t4545.400} lists all the primitive roots $z$ modulo $p=89$ and the corresponding simultaneous pairs, if a pair $z$ and $f(z)=(z+2)(z+1)^2$ exists, and Figure \ref{f4545.400} displays the spectrum of simultaneous pairs of primitive roots.

}
\end{exa}

\begin{table}[h!]
\centering
\caption{Data for the prime finite field $\F_{89}$} \label{t4545.400}
\begin{tabular}{c|c|c| c|c|c|c| c}
$z$&$f(z)=(z+2)(z+1)^2$&$z$&$f(z)$&$z$&$f(z)$&$z$&$f(z)$\\
\hline
3&  0&     27   &41    &46& 33&   63   &41\\
6&    0&   28   &43    &48  & 0&  65     &0\\
7&  0&    29   &43     &51  &0&   66&    0 \\
13&   3&   30   &0   &54  &33&   70&     0 \\
14&   0&   31   &61  &56&  29 &  74&     33\\
15&  0&    33   &54   &58& 66 &   75   &19\\
19&  0&    35   &70    &59& 0 &   76  &0\\
23&  0&    38   &0  &60& 14 &     82&   0\\
24&  0&    41   &24  &61& 3 &      83&0\\
26&  31&    43   &0  &62& 0 &     86&0\\

\end{tabular}
\end{table}

\begin{figure}[h]
\caption{The spectrum of the pair $(z,f(z))$ of simultaneous primitive roots in $\F_{89}$} \label{f4545.400}\centering%
  \begin{tikzpicture}
	\begin{axis}[
		xlabel=$z$,
		ylabel=$f(z)$,
width=0.95\textwidth,
       height=0.5\textwidth		]
	\addplot[only marks] coordinates {
		(3,0)(6,0)(7,0)(13,3)(14,0)(15,0)(19,0)
(23,0)(24,0)(26,31)(27,41)(28,43)(29,43)
(30,0)(31,61)(33,54)(35,70)(38,0)
(41,24)(43,0)(46,33)(48,0)
(51,0)(54,33)(56,29)(58,66)(59,0)
(60,14)(61,3)(62,0)
(63,41)(65,0)(66,0)(70,0)(74,33)(75,19)
(76,0)(82,0)(83,0)(86,0)
		};
	\end{axis}
\end{tikzpicture}	
\end{figure}
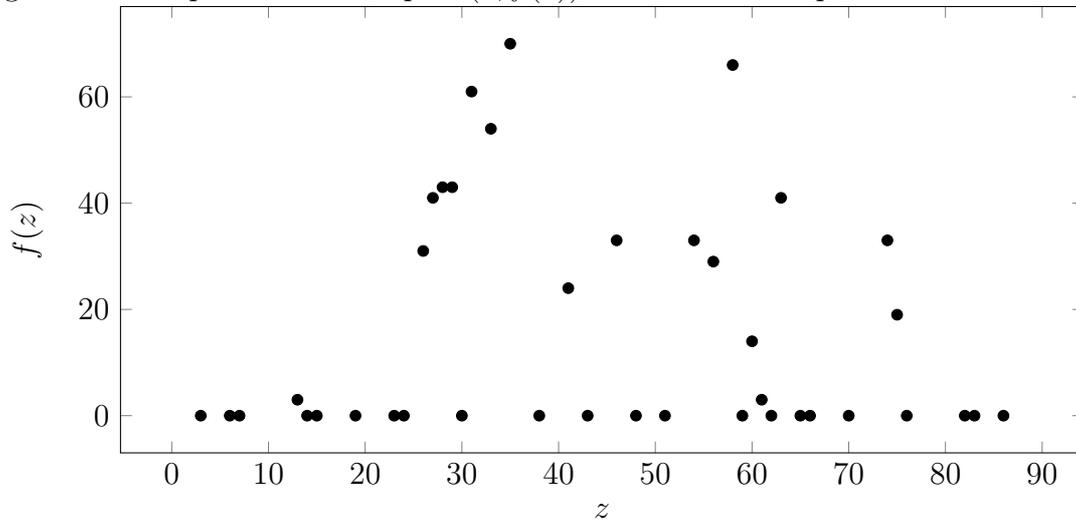

\section{Problems}\label{EXE3535}

\subsection{Primitive Root Algorithm Problems}\label{EXE3520}
\begin{exe}\label{exe3520.010} {\normalfont Develop a divisor-free primitive root test. The standard primitive root test is totally dependent on the divisors of $p-1$.
}
\end{exe}

\subsection{Prime Primitive Roots Problems}\label{EXE3550}

\begin{exe}\label{exe3550.010} {\normalfont Let $p\geq 2$ be a large prime. Estimate the number $U(p)\geq1$ of prime primitive roots in the subset $$\mathcal{U}(p)=\{q<p:q \text{ is prime and }\ord_pq=p-1\}.$$
Prove or disprove that $U(p) \asymp \pi(p)$, where $\pi(x)=\#\{p\leq x:p \text{ is prime}\}$ is the usual prime counting function.
}
\end{exe}

\subsection{Density and Distribution Problems}\label{EXE3530}
\begin{exe}\label{exe3530.020} {\normalfont Let $f(t)=x^2+1$. Compute the exact and numerical approximation for the density constant $c_f>0$ appearing in Theorem \ref{thm1210.200}, see \cite{MK1983}, \cite{PF1997}, etc., for some related theoretical background.
}
\end{exe}

\begin{exe}\label{exe3530.030} {\normalfont Let $p$ be a large prime, and let $f(t)=x^2+1$. The distribution of simultaneous pairs of primitive roots $z$ and $f(z)$ seems to be symmetric around $(p-1)/2$, see Examples \ref{exa4545.100} and \ref{exa4545.200}. Determine or estimate the distribution of simultaneous pairs of primitive roots $z$ and $f(z)$ as $p\to \infty$,
}
\end{exe}

\subsection{Average Order Problems}\label{EXE3540}

\begin{exe}\label{exe3540.010} {\normalfont Let $p\geq 2$ be a large prime, let $x\leq p$, and let $\mathcal{T}\subset \F_p$ be the subset of primitive roots modulo $p$. If $d(n)=\sum_{d\mid n}1$ is the divisors function, compute or estimate the following restricted average orders.
\begin{enumerate}[label=\alph*)]
\item $\displaystyle \sum_{n\in \mathcal{T}}d(n),$ the complete finite sum over the subset of primitive roots.\\
\item $\displaystyle \sum_{n\leq x,\; n\in \mathcal{T}}d(n),$ the incomplete finite sum over the subset of primitive roots.\\
\end{enumerate}
The complete sum is expected to be simpler than the incomplete sum.
}
\end{exe}

\begin{exe}\label{exe3540.020} {\normalfont Let $p\geq 2$ be a large prime, let $x\leq p$, and let $\mathcal{T}\subset \F_p$ be the subset of primitive roots modulo $p$. If $\sigma(n)=\sum_{d\mid n}d$ is the sum of divisors function, compute or estimate the following restricted average orders.
\begin{enumerate}[label=\alph*)]
\item $\displaystyle \sum_{n\in \mathcal{T}}\sigma(n),$ the complete finite sum over the subset of primitive roots.\\
\item $\displaystyle \sum_{n\leq x,\; n\in \mathcal{T}}\sigma(n),$ the incomplete finite sum over the subset of primitive roots.\\
\end{enumerate}
The complete sum is expected to be simpler than the incomplete sum.
}
\end{exe}

\begin{exe}\label{exe3540.030} {\normalfont Let $p\geq 2$ be a large prime, let $x\leq p$, and let $\mathcal{T}\subset \F_p$ be the subset of primitive roots modulo $p$. If $\varphi(n)=n\sum_{d\mid n}\mu(d)/d$ is the totient function, compute or estimate the following restricted average orders.
\begin{enumerate}[label=\alph*)]
\item $\displaystyle \sum_{n\in \mathcal{T}}\varphi(n),$ the complete finite sum over the subset of primitive roots.\\
\item $\displaystyle \sum_{n\leq x,\; n\in \mathcal{T}}\varphi(n),$ the incomplete finite sum over the subset of primitive roots.\\
\end{enumerate}
The complete sum is expected to be simpler than the incomplete sum.
}
\end{exe}

\subsection{Value Set Problems}\label{EXE3550}
\begin{exe}\label{exe3540.010} {\normalfont Let $p\geq 2$ be a large prime, and let $\mathcal{T}\subset \F_p$ be the subset of primitive roots modulo $p$. Estimate the number of primitive roots $$V_{d}(p)=\#\mathcal{V}_{d}(p)$$ in the value set $$\mathcal{V}_{d}(p)=\{ \tau: d(\tau) \text{ is a primitive root}\}.$$
}
\end{exe}

\begin{exe}\label{exe3540.020} {\normalfont Let $p\geq 2$ be a large prime, and let $\mathcal{T}\subset \F_p$ be the subset of primitive roots modulo $p$. Estimate the number of primitive roots $$V_{\sigma}(p)=\#\mathcal{V}_{\sigma}(p)$$ in the value set $$\mathcal{V}_{\sigma}(p)=\{ \tau: \sigma(\tau) \text{ is a primitive root}\}.$$
}
\end{exe}

\begin{exe}\label{exe3540.030} {\normalfont Let $p\geq 2$ be a large prime, and let $\mathcal{T}\subset \F_p$ be the subset of primitive roots modulo $p$. Estimate the number of primitive roots $$V_{\varphi}(p)=\#\mathcal{V}_{\varphi}(p)$$ in the value set $$\mathcal{V}_{\varphi}(p)=\{ \tau: \varphi(\tau) \text{ is a primitive root}\}.$$
}
\end{exe}

\newpage

\currfilename.\\

\end{document}